\documentclass[12pt,a4paper]{scrartcl}

\usepackage[T1]{fontenc}
\usepackage{textcomp} 
\usepackage{calc}

\usepackage{setspace} 
\usepackage[leqno]{amsmath}
\usepackage{amsfonts,amssymb,amsthm}

\usepackage[all]{xy}
\SelectTips{cm}{12}   

\usepackage{relsize}  

\newcommand{\Cone}{\mathrm{cone}}

\newcounter{lemma}
\theoremstyle{plain}                        
\newtheorem{localtheorem}[lemma]{Theorem}
\newtheorem*{localtheoremnn}{Comparison Theorem}
\newtheorem{locallemma}[lemma]{Lemma}
\newtheorem{localproposition}[lemma]{Proposition}

\newtheorem{localconstruction}[lemma]{Construction}

\theoremstyle{definition}

\newtheorem*{localremark}{Remark}
\newtheorem*{localdefinition}{Definition}

\newcommand{\Coh}{\mathrm{Coh}}
\newcommand{\Db}{\mathrm{D^b}}

\newcommand{\FM}{\mathsf{FM}}
\newcommand{\TT}{\mathsf{T}\hspace{-2pt}}

\newcommand{\sym}{{\mathrm{sym}}}

\DeclareMathOperator{\Pic}{Pic}

\DeclareMathOperator{\Quot}{Quot}
\DeclareMathOperator{\Hilb}{Hilb}

\newcommand{\Sym}{\mathrm{Sym}}

\DeclareMathOperator{\Ext}{Ext}
\DeclareMathOperator{\Hom}{Hom}

\DeclareMathOperator{\rk}{rk}

\DeclareMathOperator{\id}{id}

\newcommand{\dual}{^\vee}
\newcommand{\ddual}{^{\vee\vee}}
\newcommand{\inv}{^{-1}}
\newcommand{\orth}{^\perp}
\newcommand{\rarpa}[1]{\stackrel{#1}{\rightarrow}}

\newcommand{\isom}{ \text{{\hspace{0.48em}\raisebox{0.8ex}{${\scriptscriptstyle\sim}$}}}
                    \hspace{-0.65em}{\rightarrow}\hspace{0.3em}} 
\newcommand{\embed}{\hookrightarrow}

\renewcommand{\implies}{\Rightarrow}

\newcommand{\IN}{\mathbb N}

\newcommand{\IP}{\mathbb P}
\newcommand{\IQ}{\mathbb Q}

\newcommand{\IZ}{\mathbb Z}

\newcommand{\kg}{\mathcal{G}}

\newcommand{\ki}{\mathcal{I}}

\newcommand{\km}{\mathcal{M}}

\newcommand{\ko}{\mathcal{O}}
\newcommand{\kp}{\mathcal{P}}

\newcommand{\ks}{\mathcal{S}}
\newcommand{\kt}{\mathcal{T}}

\renewcommand{\tilde}[1]{\widetilde{#1}}

\newcommand{\coker}{\mathrm{coker}}

\newcommand{\im}{\mathrm{im}}
\newcommand{\supp}{\mathrm{supp}}

\newcommand{\map}[1]{\stackrel{#1}{\longrightarrow}}

\newcommand{\bib}[4]{\bibitem{#1} #2: \emph{#3}, #4.\vspace{-1ex}}

\newcounter{abccounter}
\newenvironment{abcliste}{\begin{list}{(\alph{abccounter})}
                      {\usecounter{abccounter}
                       \setlength{\topsep}{0ex}
                       \setlength{\partopsep}{0ex}
                       \setlength{\listparindent}{0ex}
                       \setlength{\itemsep}{0ex}
                       \setlength{\parsep}{0ex}
                       \setlength{\leftmargin}{3em}
                       \setlength{\labelwidth}{2em}
                       \setlength{\parskip}{0ex}
                      }
                      }{\end{list}}

\newcommand{\PH}{\IP} 

\begin{document}

\begin{center}
\textbf{\Large Postnikov-Stability versus Semistability of
Sheaves}

\bigskip

Georg Hein\footnote{
Universit\"at Duisburg-Essen, FB Mathematik D-45117 Essen,
\texttt{georg.hein@uni-due.de}}
, David Ploog\footnote{
Leibniz-Universit\"at Hannover, Welfengarten, D-30167 Hannover,
\texttt{ploog@math.uni-hannover.de}}
\\
\today
\end{center}

\begin{quote}{\small\scshape Abstract}
We present a novel notion of stable objects in a triangulated category.
This Postnikov-stability is preserved by equivalences.
We show that for the derived category of a projective
variety this notion includes the case of semistable sheaves.
As one application we compactify a moduli space of stable bundles
using genuine complexes via Fourier-Mukai transforms.\\
{\bf MSC 2000:} 14F05, 14J60, 14D20
\end{quote}

\subsection*{Introduction}

Let $X$ be a polarized, normal projective variety of dimension $n$ over an
algebraically closed field $k$. Our aim is to introduce a stability notion
for complexes, i.e.\ for objects of $\Db(X)$, the bounded derived category
of coherent sheaves on $X$. There are two main motivations for this notion:
on the one hand, Falting's observation that semistability on curves can be
phrased as the existence of non-trivial orthogonal sheaves \cite{Faltings}
and on the other hand, the recent proof of \'Alvarez-C\'onsul and King that
every Gieseker semistable sheaf possesses a non-trivial orthogonal object,
regardless of dimension \cite{ACK}. This result together with the
homological sheaf condition (Proposition
\ref{proposition-sheaf-conditions}) and the homological criterion for
purity (Proposition \ref{proposition-purity-conditions}) yields a purely
homological condition (Theorem \ref{P-implies-Gieseker}) for a complex to
be isomorphic to a Gieseker semistable sheaf of given Hilbert polynomial.

It seems only fair to point out that the results of this article in all
probability bear no connection with Bridgeland's notion of t-stability on
triangulated categories (see \cite{bridgeland}). His starting point about
(semi)stability in the classical setting is the Harder-Narashiman filtration
whereas, as mentioned above, we are interested in the possibility to capture
semistability in terms of Hom's in the derived category. Our approach is much
closer to, but completely independent of, Inaba (see \cite{inaba}).

\minisec{On notation:} We deviate slightly from common usage by writing $e^i$
for the $i$-th cohomology sheaf of an object $e\in\Db(X)$. Derivation of
functors is not denoted by a symbol: e.g.\ for a proper map $f:X\to Y$, we denote
by $f_*:\Db(X)\to\Db(Y)$ the exact functor obtained by deriving
$f_*:\Coh(X)\to\Coh(Y)$.

Given objects $a$, $b$ of a $k$-linear triangulated category, set
$\Hom^i(a,b):=\Hom(a,b[i])$ and $\hom^i(a,b):=\dim_k\Hom^i(a,b)$.
For $e\in\Db(X)$, we put $H^i(e):=\Hom^i(\ko_X,e)$ and $h^i(e):=\dim H^i(e)$.
The Hilbert polynomial of $e$ is denoted by $p_e$; it is defined by
$p_e(l)=\chi(e(l)):=\sum_i(-1)^i h^i(e\otimes\ko_X(l))$.
If $Z\subset X$ is a closed subset, then $e|_Z:=e\otimes\ko_Z$ denotes the
derived tensor product. For a line bundle $L$ on $X$, the notation $L^n$ will
mean the $n$-fold tensor product of $L$, except for the trivial bundle, where
$\ko_X^n$ denotes the free bundle of rank $n$.


\subsection*{P-stability}

Let $\kt$ be a $k$-linear triangulated category for some field $k$; we think
of $\kt=\Db(X)$, the bounded derived category of a normal projective variety
$X$, defined over an algebraically closed field $k$. A \emph{Postnikov-datum}
or just \emph{P-datum} is a finite collection
$C_d, C_{d-1},\dots, C_{e+1}, C_e\in\kt$ of objects together with nonnegative
integers $N_j^i$ (for $i,j\in\IZ$) of which only a finite number are nonzero.
We will write $(C_\bullet,N)$ for this.

Recall the notions of Postnikov system and convolution (see \cite{GM},
\cite{BBD}, \cite{Orlov}, \cite{Kawamata}): given finitely many objects $A_i$
(suppose $n\geq i\geq 0$) of $\kt$ together with morphisms
$d_i:A_{i+1}\to A_{i}$ such that $d^2=0$, a diagram of the form
\[ \xymatrix@=1em{
  A_n    \ar[rr]^{d_{n-1}} \ar@{=}[ddr] & &
  A_{n-1} \ar[rr]^{d_{n-2}} \ar[ddr] & &
  A_{n-2} \ar[r]           \ar[ddr] &
  \cdots\cdots & & \cdots\cdots\ar[r] &
  A_1 \ar[rr]^{d_{0}}      \ar[ddr] & &
  A_0                     \ar[ddr] \\ \\
&  T_n    \ar[uur] & &
  T_{n-1} \ar[uur] \ar[ll]^{[1]} & &
  T_{n-2}               \ar[ll]^{[1]} &
  \cdots\cdots \ar[l] &
  T_2    \ar[uur]      \ar[l] & &
  T_{1}   \ar[uur]      \ar[ll]^{[1]} & &
  T_{0}                 \ar[ll]^{[1]}
} \]
(where the upper triangles are commutative and the lower ones are distinguished)
is called a \emph{Postnikov system} subordinated to the $A_i$ and $d_i$. The
object $T_0$ is called the \emph{convolution} of the Postnikov system.

\begin{localdefinition}
An object $A\in\kt$ is \emph{P-stable with respect to $(C_\bullet,N)$} if
\begin{abcliste}
\item[(i)]
      $\hom^i_{\kt}(C_j,A)=N^i_j$ for all $j=d,\dots,e$ and all $i$.
\item[(ii)]
      For $j>0$, there are morphisms $d_j:C_j\to C_{j-1}$ such that $d^2=0$ and
      that the complex $(C_{\bullet\ge0},d_\bullet)$ admits a convolution $K$.
\item[(iii)]
      $K\in A\orth$, i.e.\ $\Hom_{\kt}^*(A,K)=0$.
\end{abcliste}
\end{localdefinition}

\begin{localremark}
$\text{\phantom{xxx}}$
\begin{abcliste}
\item Convolutions in general do not exist, and if they do, there is no
      uniqueness in general, either. There are restrictions on the
      $\Hom^i(C_a,C_b)$'s which ensure the existence of a (unique) convolution.
      For example, if $\kt=\Db(X)$ and all $C_j$ are sheaves, then the unique
      convolution is just the complex $C_\bullet$ considered as an object of $\Db(X)$.
\item Note that the objects $C_j$ with $j<0$ do not take part in forming the
      Postnikov system. We call the conditions enforced by these objects
      via (i) the \emph{passive} stability conditions. They can be used to
      ensure numerical constraints, like fixing the Hilbert polynomial of
      sheaves. In some cases, it is useful to specify only some of the $N^i_j$.
      We will do this a few times --- the whole theory runs completely parallel,
      with a slightly more cumbersome notation.
\item In many situations there will be trivial choices that ensure P-stability.
      This should be considered as a defect of the parameters (like choosing
      non-ample line bundles when defining $\mu$-stability) and not as a defect
      of the definition.
\end{abcliste}
\end{localremark}

By the very definition of $P$-stability, the following statement about
preservation of stability under fully faithful functors (e.g.\ equivalences)
is immediate.

\begin{localtheorem} \label{theorem-preservation}
Let $\Phi:\kt \to \ks$ be an exact, fully faithful functor between $k$-linear
triangulated categories $\kt$ and $\ks$, and $(C_\bullet,N)$ a $P$-datum
in $\kt$. Then, an object $A \in \kt$ is P-stable with respect to $(C_\bullet,N)$
if and only if $\Phi(A)$ is P-stable with respect to $(\Phi(C_\bullet),N)$.
\end{localtheorem}

This shifts the viewpoint from preservation of stability to transformation
of stability parameters under Fourier-Mukai transforms. See Proposition
\ref{spanish} for an example. The main result of this article is the following
theorem: P-stability contains both Gieseker stability and $\mu$-stability.

\begin{localtheoremnn} Let $X$ be a smooth projective variety and $H$ a very ample
divisor on $X$. Fix a Hilbert polynomial $p$. Then there is a P-stability datum
$(C_\bullet, N)$ such that for any object $E\in\Db(X)$
the following conditions are equivalent:
\begin{abcliste}
\item[(i)]  $E$ is a $\mu$-semistable sheaf with respect to $H$
             of Hilbert polynomial $p$
\item[(ii)] $E$ is P-stable with respect to $(C_\bullet, N)$.
\end{abcliste}

\noindent
Likewise, there is a P-stability datum $(C'_\bullet, N')$ such that for any
object $E\in\Db(X)$ the following conditions are equivalent:
\begin{abcliste}
\item[(i')]  $E$ is a Gieseker semistable pure sheaf with respect to $H$
             of Hilbert polynomial $p$
\item[(ii')] $E$ is P-stable with respect to $(C'_\bullet, N')$.
\end{abcliste}
\end{localtheoremnn}

The proof of this theorem occupies the next section. The actual statements are
slightly sharper; see Theorems \ref{P-implies-mu} and \ref{P-implies-Gieseker}.
The case of surfaces was already treated in \cite{HP2}.


\section{Proof of the Comparison Theorem}

The proof proceeds in the following steps: \\
1. Euler triangle and generically injective morphisms. \\
2. Homological conditions for a complex to be a sheaf. \\
3. Homological conditions for purity of a sheaf. \\
4. Homological conditions for semistability on curves. \\
5. P-stability implies $\mu$-semistability. \\
6. P-stability implies Gieseker semistability.

\subsection{The Euler triangle}

\begin{locallemma}\label{lemma-generic-injective}
Let $U$ and $W$ be $k$-vector spaces of finite dimension. Consider a morphism
 $\rho:U\otimes\ko_{\IP^n}\to W\otimes\ko_{\IP^n}(1)$
with nonzero kernel $K=\ker(\rho)$. Then for any integer $m\geq(\dim(U)-1)n$ we have $H^0(K(m))\neq0$.
\end{locallemma}
\proof
Denoting $I:=\im(\rho)$ and $C:=\coker(\rho)$, there are two short exact sequences
$0\to K\to U\otimes\ko_{\IP^n}\to I\to0$ and
$0\to I\to W\otimes\ko_{\IP^n}(1)\to C\to 0$.
Their long cohomology sequences yield
 $h^0(K(k)) \geq h^0(U\otimes\ko_{\IP^n}(k))-h^0(I(k))$ and
 $h^0(I(k)) \leq h^0(W\otimes\ko_{\IP^n}(1))$.

First assume $\dim(W) < \dim(U)$. This implies
$h^0(I(m))\leq(\dim(U)-1)\binom{n+1+m}{n}$.
Since $h^0(U\otimes\ko_{\IP^n}(m))=\dim(U)\binom{n+m}{n}$,
this yields
$h^0(U\otimes\ko_{\IP^n}(m))>h^0(I(m))$ for all $m \geq
(\dim(U)-1)n$. Thus, we obtain $h^0(K(m))>0$ for $m \geq (\dim(U)-1)n$.

Now assume $\dim(W) \geq \dim(U)$. Then $C$ has rank at least
$\dim(W)-\dim(U)+1$. Hence there exists a subspace $W'\subset W$ of
dimension $\dim(W)-\dim(U)+1$ such that the resulting morphism
$W'\otimes\ko_{\IP^n}(1)\to C$ is injective in the generic point, and
hence injective. Thus, the image of the injective morphism
$H^0(I(k)) \to H^0(W\otimes\ko_{\IP^n}(k+1))$ is transversal to
$H^0(W'\otimes\ko_{\IP^n}(k+1))$. This implies
$h^0(I(m))\leq(\dim(U)-1)\binom{n+1+m}{n}$ and we proceed as before.
\qed

\begin{localconstruction}\label{Sm}
The Euler triangle and objects $S^m(V,a,b)$.
\end{localconstruction}
For any two objects $a,b$ of a $k$-linear triangulated category $\kt$ and some
subspace $V\subset\Hom(a,b)$ of finite dimension we define a distinguished (Euler)
triangle
\[ S^m(V,a,b) \to \Sym^{m+1}(V)\otimes a \map{\theta} \Sym^m(V)\otimes b
              \to S^m(V,a,b)[1] \]
where tensor products of vector spaces and objects are just finite direct sums,
and $\theta$ is induced by the natural map
 $\Sym^{m+1}(V)\to\Sym^m(V)\otimes\Hom(a,b)$,
 $f_0\vee\dots\vee f_m \mapsto
  \sum_i (f_0\vee\cdots\hat{f_i}\cdots\vee f_m)\otimes f_i$.
If $\Hom(a,b)$ is finite dimensional, we use the short hand
$S^m(a,b):=S^m(\Hom(a,b),a,b)$.
For any $c\in\kt$, there is a long exact sequence
\[\xymatrix@C=1.4em
{  \Hom^{k-1}(b,c)\otimes\Sym^m(V\dual)\ar[r]
 & \Hom^{k-1}(a,c)\otimes\Sym^{m+1}(V\dual) \ar[r]
 & \Hom^{k-1}(S^m(V,a,b),c) \ar[dll]
\\
   \Hom^{k}(b,c)\otimes\Sym^m(V\dual)\ar[r]
 & \Hom^{k}(a,c)\otimes\Sym^{m+1}(V\dual) \ar[r]
 & \Hom^{k}(S^m(V,a,b),c).
} \]

\begin{localremark}
In the special case where $\kt=\Db(\IP^n_k)$ is the bounded derived
category of the projective space $\IP^n_k$ and $a=\ko_{\IP^n}$,
$b=\ko_{\IP^n}(1)$, $V=\Hom(a,b)=H^0(\ko_{\IP^n}(1))$ and $m=0$, the above
triangle comes from  the Euler sequence
 $0\to\Omega_{\IP^n}(1)\to\ko_{\IP^n}^{n+1} \to\ko_{\IP^n}(1)\to 0.$
\end{localremark}

\begin{locallemma}\label{lemma-injective-sm}
Let $\kt$ be a triangulated $k$-linear category with finite-dimensional Hom's,
$a,b,c\in\kt$ objects with $\Hom^{-1}(a,c)=0$ and let $V\subset\Hom(a,b)$ be
a subspace. Then the following conditions are equivalent:
\begin{abcliste}
\item[(i)]  The natural morphism $\varrho_v:\Hom(b,c)\to\Hom(a,c)$ is injective for
            general $v \in V$.
\item[(ii)] $\Hom^{-1}(S^m(V,a,b),c)=0$ holds for some
            $m\geq(\dim(V)-1)(\hom(b,c)-1)$.
\end{abcliste}
\end{locallemma}
\proof
We consider the morphism $\Hom(b,c)\to V\dual\otimes\Hom(a,c)$.
Together with the natural surjection
 $V\dual\otimes\ko_{\IP(V\dual)}\to\ko_{\IP(V\dual)}(1)$,
this gives a morphism
\[ \varrho: \Hom(b,c) \otimes \ko_{\IP(V\dual)} \to
            \Hom(a,c) \otimes \ko_{\IP(V\dual)}(1) \quad
   \mbox{ on} \quad \IP(V\dual) \,.\]
The injectivity of $\varrho$ is equivalent to the injectivity at all
stalks, i.e.\ of $\varrho_v:\Hom(b,c)\to\Hom(a,c)$ for all $v\in V$;
since $\ker(\varrho)$ is a subsheaf of a torsion free sheaf, the
injectivity for just one $v \in V$ is enough.
By Lemma \ref{lemma-generic-injective} this is equivalent to the injectivity of
\[ H^0( \varrho \otimes \ko_{\IP(V\dual)}(m)) :
   H^0(\Hom(b,c) \otimes \ko_{\IP(V\dual)}(m)) \to
   H^0(\Hom(a,c) \otimes \ko_{\IP(V\dual)}(m+1)) \]
for $m = (\dim(V)-1)(\hom(b,c)-1)$.
Since $\Hom^{-1}(a,c)=0$, the long exact cohomology sequence of the
triangle from Construction \ref{Sm} gives that the kernel of
$H^0(\varrho\otimes\ko_{\IP(V\dual)}(m))$ is $\Hom^{-1}(S^m(V,a,b),c)$.
\qed


\subsection{Sheaf conditions} \label{Sheaf-Conditions}

Let $X$ be a projective variety over $k$ (in this subsection, we only need
$k$ to be infinite) and $\ko_X(1)$ a line bundle corresponding to the very
ample divisor $H$. Let $V=H^0(\ko_X(1))$ the space  of global sections and
$\PH:=\IP(V\dual)=|H|$ the complete linear system for $H$.

Our aim is to find conditions on a complex $e\in\Db(X)$ in terms of the Hom's
from finitely many test objects, ensuring that $e$ is isomorphic to a sheaf,
i.e.\ a complex concentrated in degree 0. These conditions only depend on the
Hilbert polynomial $p_e$ with respect to $\ko_X(1)$.

\minisec{The numerical data}
Fix non-negative integers $n$ and $v$. For a polynomial function $p\in\IQ[t]$
with integer values, its derivative is defined as $p'(t):=p(t)-p(t-1)$. We also
set $\sym_v(m):=\binom{m+v-1}{v-1}$, which is the dimension of $\Sym^m(V)$ for
a $v$-dimensional vector space $V$.

Call a sequence $(m_1,\dots,m_n)$ of integers \emph{$(p,n)$-admissible} if
$m_{k+1}\geq(p_k(-l)-1)(v-1)$ for $l=1,\dots,n-k$ where the  polynomials $p_0$,
\dots,$p_{n-1}$ are defined by $p_0=p$ and
 $p_{k+1}=\sym_v(m_{k+1})\cdot p_k' + \sym_{v-1}(m_{k+1})\cdot p_k$.
One can easily define a $(p,n)$-admissible sequence by recursion: set
$m_{k+1}:=\max\{(p_k(-l)-1)(v-1) \:|\: l=1,\dots,n-k\}$, the polynomials
being defined by the above formula in each step.

Suppose that $(m_1,\dots,m_n)$ is a $(p,n$)-admissible sequence and that
$p_k(-l)\geq0$ for all $l,k\geq0$ with $l+k\leq n$. Then $(m_1,\dots,m_{n-1})$
is a $(p',n-1)$-admissible sequence, as follows from induction and unwinding
the definitions. In this case, if the auxiliary polynomials for the
$(p,n)$-sequence are denoted $p_0,\dots,p_n$ as above, then those for the
$(p',n-1)$-sequence are just $p_0',\dots,p_{n-1}'$.

\minisec{The vector bundles $G_m$ and $S_m$ and $F_k$}
We denote the standard  projections by $p:\PH\times X\to\PH$ and
$q:\PH\times X\to X$. The identity in
 $V\otimes V\dual = H^0(\ko_X(H))\otimes H^0(\ko_{\PH}(1))
                  = \Hom(q^*\ko_X(-H),p^*\ko_{\PH}(1))$
yields a natural morphism
 $\alpha:q^*\ko_X(-H)\to p^*\ko_{\PH}(1)$.
The cokernel $\kg$ of $\alpha$ is the universal divisor, i.e.\
$\kg|_{\{D\}\times X}=\ko_D$ for all $D\in\PH$.
We can consider $q^*\ko_X(-H)$ and $p^*\ko_\PH(1)$ and $\kg$ as Fourier-Mukai
kernels on $\PH\times X$. Then we obtain, for any object $a\in\Db(\PH)$, an
exact triangle $\FM_{q^*\ko_X(-H)}(a)\to\FM_{p^*\ko_\PH(1)}(a)\to\FM_\kg(a)$. In
particular, we set $G_m:=\FM_\kg(\ko_\PH(m))$. The projection formula and base
change show that the above triangle reduces to the short exact sequence
 $0\to\Sym^m(V)\otimes\ko_X(-H)\to\Sym^{m+1}(V)\otimes\ko_X\to G_m\to0$
if $m>-n$. Hence in this case, $G_m=R^0q_*(\kg\otimes p^*\ko_{\PH}(m))$ is a
vector bundle and the higher direct images vanish. The exact sequence also
yields $p_{G_m\otimes e}=\sym_v(m)\cdot p_e' + \sym_{v-1}(m)\cdot p_e$.

Let $S_m:=G_m\dual$; note that $S_m=S^m(V,\ko_X,\ko_X(1))$ using Construction
\ref{Sm}. By  Lemma \ref{lemma-injective-sm}, for $e\in\Db(X)$ with
$H^{-1}(e)=0$ and $m\geq(v-1)(h^0(e(-1))-1)$, the following conditions are equivalent:

\noindent
\begin{tabular}{ll}
(i)   & $H^{-1}(e|_D)=0$ for general $D\in|H|$ with $e|_D:=e\otimes\ko_D$
        (derived tensor product)\\
(ii)  & $\Hom^{-1}(S_m,e)=0$ \\
(iii) & $H^{-1}(e\otimes G_m)=0$.
\end{tabular}

\noindent
Finally, we define another series of vector bundles by $F_0:=\ko_X$ and
$F_k:=F_{k-1}\otimes G_{m_k}$.

\begin{localproposition} \label{proposition-sheaf-conditions}
Let $X$ be a projective variety of dimension $n$ and $\ko_X(1)$ a
very ample line bundle. Let $V=H^0(\ko_X(1))$ and $v=\dim(V)$. Let
$p\in\IQ[t]$ be an integer valued polynomial with $\deg(p)\leq n$.
Suppose that $(m_1,\dots,m_n)$ is a $(p,n)$-admissible sequence with auxiliary
polynomials $p_1$,\dots,$p_n$.

Assume that $e\in\Db(X)$ is an object such that for all $l,k\geq0$ with
$l+k\leq n$ we have $h^0(F_k(-l)\otimes e)=p_k(-l)$ and
$h^i(F_k(-l)\otimes e)=0$ for all $i\neq0$. Then $e\cong e^0$ is a sheaf with
Hilbert polynomial $p$. Furthermore, $e$ is 0-regular.
\end{localproposition}

\begin{proof}
We proceed by induction on the dimension $n$. The start $n=0$ is trivial.

Let $n>0$. In a first step, we use induction to show that the complex $e|_D$ is
a sheaf for general $D\in|H|$. Let us begin by pointing out that
$(m_1,\dots,m_{n-1})$ is a $(p',n-1)$-admissible sequence. Next, the graded
vector spaces $H^*(F_k(-l)\otimes e)$ vanish by assumption outside of
degree 0, where $k,l\geq0$ and $k+l\leq n$. Hence, $H^*(F_k(-l)\otimes e|_D)$
can be nontrivial at most in degrees 0 and $-1$ (where $k+l<n$). But
 $H^{-1}(F_k(-l)\otimes e)=0$ and
 $H^{-1}(G_{m_{k+1}}\otimes F_k(-l)\otimes e)=0$ and
 $m_{k+1} \geq (v-1)(p_k(-l-1)-1)$
together with $p_k(-l-1)=h^0(F_{k}(-l-1)\otimes e)$
imply $H^{-1}(e|_D)=0$ for general $D\in|H|$ by Lemma \ref{lemma-injective-sm}.
Thus, $H^*(F_k(-l)\otimes e|_D)$ is concentrated in degree 0 and of the correct
dimension $h^0(F_k(-l)\otimes e|_D)=p_k(-l)-p_k(-l-1)=p_k'(-l)$.

We fix a smooth $D$ such that $e|_D$ is a sheaf. Therefore, homology sheaves
$e^i$ in degrees $i\neq0$ are either zero or have zero-dimensional support.
(Support of dimension one or higher would be detected by a general $D\in|H|$.)
Looking at the Eilenberg-Moore spectral sequence
\[ E_2^{p,q} = \Ext_X^q(\ko_X(k),e^{-p}) \Rightarrow H^{p+q}(e(-k)), \]
we see that it has non-zero $E_2$ terms at most in the row $q=0$ and the
column $p=0$.

By induction, $e^0|_D$ is a 0-regular sheaf without higher cohomology. Then
the following piece of the long exact sequence
\[ 0=H^i(e^0|_D(k)) \to H^{i+1}(e^0(k-1)) \to H^{i+1}(e^0(k))
                    \to H^{i+1}(e^0|_D(k))=0 \]
(valid for $i\geq0$ and $k\geq-n-1$) shows
 $H^{i+1}(e^0(k-1))\cong H^{i+1}(e^0(k))$.
Since these vector spaces are zero for $k\gg0$, we see that there are no
non-trivial terms in the spectral sequence except possibly $E_2^{*,0}$ and
$E_2^{0,1}$. Hence, the only non-zero differential that might occur is
$E_2^{1,0}=H^1(e^0(k))\to E_2^{0,2}=H^0(e^{-2}(k))$. As by assumption
$H^1(e)=H^2(e)=0$, this map has to be an isomorphism. Now $e^{-2}$ is a sheaf
supported on points, so that $h^0(e^{-2}(k))$ does not depend on $k$. Hence
$h^1(e^0(k))$ does not depend on $k$ either.

From now on we proceed as in Mumford's proof on regularity
\cite[page 102]{Mum-CuOnSf}. Consider the following commutative diagram
\[ \xymatrix{
 V\otimes H^0(e^0) \ar[r]^{\alpha_0} \ar[d]^{\beta_0} &
 V\otimes H^0(e^0|_D) \ar[r]^{\alpha_1} \ar[d]^{\beta_1} &
 V\otimes H^1(e^0(-1)) \ar[d]^{\beta_2}
\\
 H^0(e^0(1)) \ar[r]^{\gamma_0} &
 H^0(e^0(1)|_D) \ar[r]^{\gamma_1} &
 H^1(e^0)
} \]
where the horizontal maps are induced from the triangles
$e^0(-1)\to e^0\to e^0|_D$ (top row) and $e^0\to e^0(1)\to e^0(1)|_D$
(bottom row) and the vertical maps correspond to composition with
$V=\Hom(\ko_X,\ko_X(1))$. The isomorphism $H^1(e^0(-1))\isom H^1(e^0)$
implies $\alpha_1=0$, hence $\alpha_0$ is surjective. Next, $e^0|_D$
is globally generated (as it is 0-regular); together with $H^1(e^0|_D)=0$
this shows that the evaluation map $\beta_1$ is surjective. Hence
$\gamma_1=0$. As we also have $H^1(e^0|_D)=0$, this is turn implies that
the natural map $H^1(e^0)\to H^1(e^0(1))$ is an isomorphism. But then
$h^1(e^0)=h^1(e^0(k))$ for all $k\gg0$, which forces $H^1(e^0)=0$.

Hence the spectral sequence has no non-zero terms outside of the row
$q=0$ and thus degenerates at the $E_2$ level. Since $E_\infty^{p+q}$
is known by assumption, this proves $H^0(e^i)=0$ for $i\neq0$. Since
these $e^i$ are supported on points, this implies $e^i=0$ and hence
$e\cong e^0$ is indeed isomorphic to a sheaf concentrated in degree 0.
The 0-regularity is an obvious consequence of the assumptions.
\end{proof}


\subsection{Purity conditions} \label{Purity-Conditions}

In this subsection, we formulate a homological purity condition for
$0$-regular sheaves on a projective variety $X$ with very ample
polarization $\ko_X(1) = \ko_X(H)$. Since this condition is needed only
for the Gieseker stability part of the Comparison Theorem, the reader
interested exclusively in slope stability may skip this subsection.

Our key result for detecting 0-dimensional subsheaves is:

\begin{locallemma}\label{pure1}
Let $E$ be a sheaf on a projective variety $X$ with very ample
polarization $\ko_X(1) = \ko_X(H)$. Let $M=h^0(E)$ and denote by
$E_0\subset E$ the maximal subsheaf of dimension zero. Then,
$E_0 = 0$ if and only if $h^0(E(-M))=0$.
\end{locallemma}

\begin{proof}
If $E_0\ne0$, then we have $h^0(E(k))\ne0 $ for all $k \in \IZ$. So we
only need to show that $h^0(E(-M)) > 0$ implies $E_0 \ne 0$. We consider
the decreasing sequence $M=h^0(E), h^0(E(-1)), \dots , h^0(E(-M))$. If
$h^0(E(-M)) > 0$, then there is an integer $k$ with
$h^0(E(-k))=h^0(E(-k-1))>0$. Let $E'$ be the image of the morphism
$H^0(E(-k))\otimes\ko_X\to E(-k)$. The sheaf $E'$ is globally generated and
satisfies the condition $h^0(E'(-1))=h^0(E')$. A general hyperplane $D\in|H|$
meets the associated locus of $E'$ transversally, and thus yields a short
exact sequence $0 \to E'(-1) \to E' \to E'|_D \to 0$. Since $E'$ is globally
generated, the sections of $E'$ also generate $E'|_D$. However, all these
sections come from $E'(-1)$. Thus $E'|_D = 0$. We conclude that the support
of $E'$ is of dimension zero.
\end{proof}

Now let $E$ be a coherent sheaf on $X$ with Hilbert polynomial $p=p_E$ of
degree $d$. Assume that $E$ is 0-regular, i.e.\ $H^i(E(-i))=0$ for $i>0$.
By \cite{Mum-CuOnSf}, this implies that $E(l)$ is globally generated for
$l\geq0$ and also that $H^i(E(l))=0$ for all $i>0$, $l\geq0$.
Set $M:=p(0)=h^0(E)$. We consider the dimension filtration of $E$
\[ 0=E_{-1} \subset E_0 \subset E_1 \subset \dots \subset E_d = E
\qquad \text{with } E_k/E_{k-1} \text{ pure of dimension }k.\]
As $E$ is globally generated, there exists a Quot scheme
$Q:=\Quot^p_{X}(\ko_X^{M})$ of finite type which parameterizes all
0-regular sheaves with Hilbert polynomial $p$.
In particular, given a coherent sheaf $F$ on $X$, there exists a universal
upper bound $B$ (depending only on $F$, $p$ and $H$) such that
$B\geq h^1(F\otimes E\otimes\ko_{H_1}\otimes\dots\otimes\ko_{H_m})$
for all $E\in Q$, $m\in\{0,\dots,d-1\}$ and $H_i\in|H|$.

\begin{localproposition} \label{proposition-purity-conditions}
Let $\ko_X(1)=\ko_X(H)$ and $p$ be as above. There exists a vector bundle
$F$ on $X$ depending only on $p$ and $\ko_X(1)$ such that for any 0-regular
sheaf $E$ on $X$ with Hilbert polynomial $p_E=p$ holds:
$E$ is pure if and only if $\Hom(F,E)=0$.
\end{localproposition}

\begin{proof}
Restriction of $E$ to a general hyperplane $H_i \in |H|$ commutes with the
dimension filtration: $(E|_{H_i})_k = E_{k+1}|_{H_i}$.
Coupled with Lemma \ref{pure1}, this shows that $E$ is pure if and only if
$H^0(E(-M) \otimes \ko_{H_1} \otimes \dots \otimes \ko_{H_m}) = 0$ for
all $m=0, \dots , d$ and general hyperplanes $H_i \in |H|$.
This condition can be checked using Lemma \ref{lemma-injective-sm}, as
done in the proof of Proposition \ref{proposition-sheaf-conditions}:
We define sequences of integers $(m_1,m_2, \dots, m_{d-1})$ and of
vector bundles $F_0, F_1, \dots , F_{d-1}$ recursively by \\
\begin{tabular}{lp{14cm}}
(i)   & $F_0 := \ko_X$\\
(ii)  & $\tilde m_k \geq
          h^1(F_{k-1} \otimes E(-M-1) \otimes \ko_{H_1}
                               \otimes \dots \otimes \ko_{H_m})$
        for all sheaves $[E] \in Q$, \\
      & all $m\in\{0,\dots d-1\}$, and all hyperplanes
       $H_i \in | H| $.\\
(iii) & $m_k=(h^0(\ko_X(1))-1)(\tilde m_k -1)$ \\
(iv)  & $F_k = F_{k-1} \otimes G_{m_k}$ where $G_{m_k}$ is the vector bundle
        from Subsection \ref{Sheaf-Conditions}.
\end{tabular}

\noindent
(We only need condition (ii) for generic hyperplanes. Note that for almost
all choices of the $H_i$, the tensor product is underived, thus just a sheaf
supported on an $m$-codimensional complete intersection.)

Proceeding as in the proof of Proposition \ref{proposition-sheaf-conditions},
the vanishing of $H^0(E(-M) \otimes F_0)$, 
\dots , $H^0(E(-M) \otimes F_{d-1})$ is equivalent to the vanishing of $H^0(E(-M))$, $H^0(E(-M)\otimes\ko_{H_1})$, \dots ,
$H^0(E(-M)\otimes \ko_{H_1} \otimes \dots \otimes \ko_{H_{d-1}})$ for general hyperplanes $H_1, H_2, \dots H_{d-1}$ in the linear system $|H|$.
By Lemma \ref{pure1}, the last condition is equivalent to $E_{d-1}=0$.
Setting $F:=(F_0\dual\oplus\cdots\oplus F_{d-1}\dual)\otimes\ko_X(M)$ yields the
required vector bundle.
\end{proof}


\subsection{Semistability on curves}

Let $X$ be a smooth projective curve of genus $g$ over $k$. Fix integers
$r>0$ and $d$. Let $\ko_X(1)$ be a fixed line bundle of degree one.

\begin{localtheorem} \label{curves}
For a coherent sheaf $E$ on $X$ of rank $r$ and degree $d$, the
following conditions are equivalent:

\begin{tabular}{lp{\linewidth-5em}}
(i)   & $E$ is a semistable vector bundle.\\
(ii)  & There is a sheaf $0\ne F$ with $E\in F\orth$, i.e.\
        $\Hom(F,E)=\Ext^1(F,E)=0$. \\
(iii) & There exists a sheaf $F$ on $X$ with
        $\det(F)\cong\ko_X(rd-r^2(g-1))$ and $\rk(F)=r^2$ such that
        $\Hom(F,E)= \Ext^1(F,E) =0$.\\
\end{tabular}
\end{localtheorem}
\begin{proof}
The equivalence (i) $\Longleftrightarrow$ (ii) is Falting's characterisation
of semistable sheaves on curves \cite{Faltings}. One direction is easy: For
 $E'\subset E$ with $\mu(E')>\mu(E)$, we have $\mu(E'\otimes F\dual)>\mu(E\otimes F\dual)$,
hence by Riemann-Roch $\chi(E'\otimes F\dual)>\chi(E\otimes F\dual)=0$. But then
 $h^0(E'\otimes F\dual)>0$, contradicting $h^0(E\otimes F\dual)=0$.
The refinement (i) $\Longleftrightarrow$ (iii) is the content of Popa's paper
\cite[Theorem 5.3]{Popa}.
\end{proof}

Based on this result, we can give two Postnikov data for semistable
bundles on $X$. Introduce the slope $\mu:=d/r$ and some further semistable vector
bundles and integers:
\begin{align*}
  A & := \ko_X((r^2+1)(\lfloor\mu\rfloor-\mu) + 2r^2(1-g) -3g), &&&
m_1 & := r^2(r^2+1)\left( \mu - \lfloor\mu\rfloor + g + 2 \right), \\
  B & := \ko_X^{r^2+1}(\lfloor\mu\rfloor -3g), &&&
m_2 & := (\hom(A,B)-1)(m_1-1) \,.
\end{align*}

\begin{localproposition} \label{curves-raynaud}
Let $X$ be a smooth projective curve $r >0$, and $d$ two integers and
$A$, $B$, $m_1$, and $m_2$ as above.
For an object $E\in\Db(X)$ the following conditions are equivalent:
\begin{tabular}{lp{\linewidth-4em}}
(i)   & $E$ is a semistable vector bundle of rank $r$ and degree $d$. \\
(ii)  & The object $E$ satisfies the following Postnikov conditions: \\
      & \begin{tabular}{ll}
        (1) & $\hom(A,E) = \hom(B,E) = m_1$,
              $\hom^i(A,E) = \hom^i(B,E) = 0$ for $i\neq0$. \\
        (2) & There is a cone $A \to B \to C$ in $\Db(X)$ with $\Hom^*(C,E)=0$. \\
        \end{tabular} \\
(iii) & The object $E$ satisfies the following Postnikov conditions: \\
      & \begin{tabular}{ll}
        (1) & $\hom(A,E) = \hom(B,E) = m_1$,
              $\hom^i(A,E) = \hom^i(B,E) = 0$ for $i\neq0$. \\
        (2) & $\Hom^{-1}(S^{m_2}(A,B),E) =0$.
        \end{tabular}
\end{tabular}
\end{localproposition}
\begin{proof}
(i)$\implies$(ii)
$E$ is semistable of degree $d$ and rank $r$, hence by Theorem \ref{curves}
there exists a sheaf $F$ with $\det(F)\cong\ko_X(rd-r^2(g-1))$ and $\rk(F)=r^2$
such that $\Hom^*(F,E)=0$. This implies that $F$ is also a semistable bundle.
Thus (see \cite[Lemma 2.1]{hein-raynaud}), it appears in a short exact
sequence $0\to A\to B\to F \to0$. Since $\mu(E^0)-\mu(A)>2g-2$, we
see that $\Hom^i(A,E)=0$ for $i\ne0$. Using the Riemann-Roch Theorem,
we deduce that $\hom(A,E)=m_1$. The same works with $B$ instead of $A$.
We eventually conclude that (1) holds.
Setting $C=F$ we obtain the object required in condition (2).

(ii)$\implies$(i) The conditions (1) and (2) imply that the morphism $A\to B$
is not zero. Since $A$ is a line bundle, this morphism is injective; hence
the distinguished triangle of (2) corresponds to a short exact sequence of
sheaves $0\to A\to B\to C \to0$. As the global dimension of a smooth curve
is one, we have $E\cong\bigoplus E^i[-i]$. The condition $\Hom^*(C,E)=0$
implies that all the $E^i$ are semistable of slope $d/r$. If $E^i\ne0$, then
$\Hom^i(A,E)\ne0$. So from condition (1) we deduce that $E$ is a sheaf object.
As the slopes of $A$ and $B$ differ, we can read off the Hilbert polynomial of
$E^0$ from the dimensions $\hom(A,E)$ and $\hom(B,E)$. Altogether, $E^0$ is of
rank $r$ and degree $d$.

(ii)$\iff$(iii) Any morphism $\alpha:A\to B$ gives a distinguished triangle
as in (ii). The total homomorphism space $\Hom^*(C,E)$ is zero if and only if
$\Hom(B,E)\to\Hom(A,E)$ is a bijection. Because we work with finite dimensional
$k$-vector spaces, this is equivalent to the injectivity of $\Hom(B,E)\to\Hom(A,E)$.
Thus, by Lemma \ref{lemma-injective-sm} we are done.
\end{proof}

For a more detailed description and the relation to the Theta divisor
and its base points see \cite[Theorems 2.12 and 3.3]{hein-raynaud} of
the first author.


\subsection{P-stability implies $\mu$-semistability}

\begin{localtheorem} \label{P-implies-mu}
{\bf(Comparison theorem for Mumford-Takemoto semistability)}\\
For a polarized normal projective Gorenstein variety $(X, \ko_X(1))$
and for a polynomial $p$ of degree $n=\dim(X)$, there exist sheaves
$C_{-m}$, $C_{-m+1}$, \dots , $C_n$, $D$
on $X$, and integers $N^i_j$ such that for an object
$E\in\Db(X)$ the following two conditions are equivalent:

\begin{tabular}{lp{14cm}}
(i)  & $E$ is a $\mu$-semistable sheaf concentrated in degree zero
       of Hilbert polynomial $p$.\\
(ii) & $\hom^i(C_j,E)= N_j^i$, and there exists a complex
       $C_\bullet = (C_n \rarpa{d_n} \dots \rarpa{d_1} C_0)$
       such that 
       $\Hom^*(C_\bullet,E)=0$, that is $E\in C_\bullet\orth$. \\
\end{tabular}
\end{localtheorem}

\begin{localremark}
Condition (ii) of the theorem states that $E$ is P-stable for the
P-datum $(C_{-m},\dots,C_n,N)$.
\end{localremark}

\begin{proof}
The objects $C_{-m},\dots,C_n$ are defined in the proof of
(i)$\implies$(ii), in a manner independent of $E$. In order to define
the sheaf $D$, pick a large number $N\gg0$ and set $Z$ to be a
fixed 2-codimensional intersection of two hyperplanes of $|H|$. 
Let $D:=\ko_Z(-N)$.

(i)$\implies$(ii)
Suppose that $E$ is a $\mu$-semistable vector bundle with given Hilbert
polynomial $p:=p_E$. As semistability implies that $E$ appears in a
bounded family, there is an integer $d_1$ (depending only on $p$) such
that $E$ is $d_1$-regular. Hence by Proposition
\ref{proposition-sheaf-conditions} there are sheaves $C_{-m}$,
$C_{-m+1}, \, \dots ,\, C_{-1}$ and integers $N^i_j$ such that
$\hom^i(C_j,E)=N^i_j$ implies that $E$ is a $d_1$-regular sheaf of
Hilbert polynomial $p$.

By Langer's effective  restriction theorem \cite[Theorem 5.2]{langer},
there is a constant $d_2$ such that $E$ is $\mu$-semistable $\iff$ the
restriction $E|_Y$ is semistable for $Y = H_1\cap H_2\cap\dots\cap H_{n-1}$
a general complete intersection of hyperplanes $H_i \in |d_2H|$. By
Bertini's theorem, we may choose $Y$ to be a smooth curve embedded by
$\iota: Y \to X$. We also choose $Y$ to be disjoint from $Z$.
The semistability of $\iota^*E$ can be expressed (see Proposition
\ref{curves-raynaud} and its proof) by
$\Hom^*(\iota^*E,F)=0$ for some coherent sheaf $F$ on $Y$ which appears in a
short exact sequence
 $0 \to \iota^*\ko_X(d_5) \to \iota^*\ko_X^{d_3}(d_4) \to F \to 0$.
Adjointness $\iota_*\dashv\iota^*$ and Serre duality together with the
Gorenstein assumption yields
 $0 = \Hom^*_Y(\iota^*E,F) = \Hom^*_X(E,\iota_*F) =
       \Hom^*_X(\omega_X\inv[-n]\otimes\iota_*F,E)^*$,
i.e.\ $E\in(\omega_X\inv\otimes\iota_*F)\orth$.

Using the Koszul complex of $\ko_Y$ and the resolution of $F$, we find
that $\omega_X\inv \otimes \iota_*F$ has a resolution
$C_n \to C_{n-1} \to \dots C_0 \to \omega_X\inv \otimes \iota_*F$
where the sheaves $C_i$ are vector bundles on $X$ of the form
 $C_i =\left( \ko_X^{b_i}(d_4-id_2) \oplus \ko_X^{a_i}(d_5-(i-1)d_2) \right)
        \otimes \omega_X\inv$.
Hence the Postnikov system $C_\bullet=(C_n \to \dots \to C_0)$ has the
convolution $C_\bullet\cong\omega_X\inv\otimes\iota_*F$, and
$E\in C_\bullet\orth$. Also by choice of $Y$ we have $C_\bullet\in D\orth$,
as a generic 1-dimensional complete intersection $Y$ will miss $Z$.

(ii)$\implies$(i)
If $E$ is a complex satisfying the conditions of (ii), then $E$ is a
$d_1$-regular sheaf by Proposition \ref{proposition-sheaf-conditions}
and the choice of $C_{-m},\dots C_{-1}$. Suppose $E$ is not $\mu$-semistable.
Assume first that the convolution $C_\bullet$ is concentrated on a
curve $Y$. Then the destabilizing subobject also destabilizes the restriction
to $Y$ and forces $\Hom^*(C,E)\ne0$ by Proposition~\ref{curves-raynaud}. If
$C$ is not concentrated on a curve, then $\Hom^*(D,C)\ne0$. However, the
general curve $Y$ will not intersect with $Z$. Thus, $\Hom^*(D,C)=0$ forces
$C$ to be concentrated on a curve (which does not intersect $Z$), and we are
done.
\end{proof}


\subsection{P-stability implies Gieseker semistability}

\begin{localtheorem} \label{P-implies-Gieseker}
{\bf(Comparison theorem for Gieseker semistability)}\\
For a polarized projective variety $(X, \ko_X(1))$ and for a given
polynomial $p$ there exist sheaves $C_{-m}$, $C_{-m+1}$, \dots , $C_0$,
$C_1$, and $F$ on $X$, and integers $N^i_j$ such that for an object
$E\in\Db(X)$ the following three conditions are equivalent:

\begin{tabular}{lp{14cm}}
(i)   & $E$ is concentrated in degree zero, and a Gieseker semistable
        sheaf of Hilbert polynomial $p$. \\
(ii)  & $\hom^i(C_j,E)= N_j^i$, $\Hom(F,E)=0$ and there exists a
        distinguished triangle $C\to C_0\to C_1\to C[1]$
        in $\Db(X)$ such that $\Hom^*(C,E)=0$, that is $E\in C\orth$. \\
(iii) & $\hom^i(C_j,E)= N_j^i$, $\Hom(F,E)=0$ and
        $\Hom^{-1}(S^m(C_0,C_1),E) =0$ for $m\gg0$.
\end{tabular}
\end{localtheorem}

\begin{proof}
(i)$\iff$(ii)
By Proposition \ref{proposition-sheaf-conditions} we can choose sheaves
$C_{-m}$, $C_{-m+1}$, \dots, $C_{-1}$ and $N_j^i\in\IN$ with $j=-m,\dots,-1$
such that any object $E$ satisfying $\hom^i(C_j,E)=N_j^i$ is a sheaf
with Hilbert polynomial $p$.
By Proposition \ref{proposition-purity-conditions} there exists a sheaf
$F$ such that $\Hom(F,E)=0$ is equivalent to the purity of $E$.

Assuming these conditions on $E$, \cite[Theorem 7.2]{ACK} implies that there
are objects $C_0, C_1\in\Db(X)$ such that the existence of the above $C$ is
equivalent to the semistability of $E$.\\
(ii)$\iff$(iii)
Here we use that the sheaves $C_0$ and $C_1$ are
direct sums of $\ko_X(-N_i)$ for $N_i \gg 0$. So $\Hom^*(C_i,E)$ is
concentrated in degree zero. Now we can argue as in  the proof of
(ii)$\iff$(iii) in Proposition \ref{curves-raynaud}.
\end{proof}

\begin{localremark}
The above system of sheaves $(F,C_{-m},\dots,C_1)$ is a P-datum.
It is worth pointing out that the active part only consists of a single
morphism, by virtue of the theorem of \'Alvarez-C\'onsul and King.

On the other hand, our treatment of the purity conditions in Subsection
\ref{Purity-Conditions} can be used to improve the statement of \cite{ACK},
as their explicit hypothesis of 'pure' sheaf can be phrased in homological
terms.
\end{localremark}


\section{Preservation of stability} \label{preservation}

The classical approach to preservation of stability is this: let $X$ and $Y$
be smooth, projective varieties and consider a moduli space $\km_X(v)$ of
semistable sheaves on $X$ with given Mukai vector $v\in H^*(X)$. If furthermore
we are given a Fourier-Mukai transform $\Phi:\Db(X)\isom\Db(Y)$, then one might
ask if a sheaf $E\in\km_X(v)$ is mapped under $\Phi$ to a shifted sheaf (i.e.\
the complex $\Phi(E)\in\Db(Y)$ has cohomology only in a single degree $i$, in
which case $E$ is called $\mathrm{WIT}_i$; the sheaf is called $\mathrm{IT}_i$
if the single cohomology sheaf is even locally free). Assuming this, one might
next wonder if the resulting sheaf on $Y$ is itself semistable with respect to
suitable numerical constraints $v'\in H^*(Y)$ and some polarization on $Y$.

The hope is to produce maps $\Phi:\km_X(v)\to\km_Y(v')$ --- a hope that is
often founded: if the Fourier-Mukai transform is of geometric origin (given
by a universal bundle, for example), then there is a plethora of results
stating that stability is preserved in this sense.

Our point is that the restriction to WIT sheaves is unnatural in the context
of derived categories. It would be much more appealing if there was a notion
of stability which is preserved by equivalences on general grounds. This would
make the classical results about preservation of stability the
special case where sheaves happen to be mapped to (shifted) sheaves again.
Our notion of P-stability provides this. The Comparison Theorem shows that
semistable sheaves in $\km_X(v)$ can be encoded via a P-datum; it is then
tautological that the objects of $\Phi(\km_X(v))$ will be P-stable with
respect to the transformed P-datum. Hence we shift our point of view to the
following question: in which cases is the transformed P-datum of classical
origin, i.e.\ induced by Gieseker or $\mu$-semistability?

\subsection{Abelian surfaces}
Here is a typical example, see \cite[Theorem 3.34]{BBH}. Let $(A,H)$ be a
polarized Abelian surface, $\hat A$ the dual Abelian surface and
$\kp\in\Pic(A\times\hat A)$ the Poincar\'e bundle. This bundle gives rise
to the classical Fourier-Mukai transform $\FM_\kp:\Db(A)\isom\Db(\hat A)$
of \cite{Mukai}. Then $\hat{H}=-c_1(\FM_\kp(\ko_A(H)))$ is a polarization
for $\hat A$.

\begin{localtheorem} \label{spanish}
If $E$ is a $\mu$-stable locally free sheaf on $A$ with $\mu(E)=0$ and
rank $r>1$, then $E$ is $\mathrm{IT}_1$ and $\FM_\kp(E)[1]$ is a
$\mu$-semistable vector bundle with respect to $\hat H$.
\end{localtheorem}

\begin{proof}
We are going to use the following characterisation of $\mu$-semistable
sheaves on an abelian surface (cf.~\cite[Theorem 3.1]{hein})
\begin{align*}
      E \text{ is $\mu$-semistable}
&\iff E\otimes\ko_C \text{ is semistable for } m\gg 0,
      \text{ and some } C\in|mH| \\
&\iff \Hom^*(E, F)=0 \text{ for some coherent sheaf } F
      \text{ on } C \text{ as above.}
\end{align*}
The first equivalence is deduced from the restriction theorem of Mehta
and Ramanathan (see \cite{MR} or also \cite{hl}, or for effective bounds
the results of Langer in \cite{langer}). The second equivalence follows
from Theorem \ref{curves}.
For $F$, we can use a torsion sheaf $F$ with a resolution by prescribed vector 
bundles, as in the proof of Theorem \ref{P-implies-mu}. Then, we have
$\Hom^*(E,F)=0$ and $\Hom^*(\ko_A,F)=0$. This defines a P-datum on $\Db(A)$. We
will show that the image under $\FM_P$ is a P-datum on $\Db(\hat A)$ containing
$\mu$-semistability for sheaves of degree 0.

For this, suppose that $\FM_\kp(E)[1]$ is a sheaf. Fix a sheaf $F$ such that
$E$ is $\mu$-semistable of degree 0 if and only if $\Hom^*(E,F)=0$. In
particular, $H^*(F)=0$ since $\ko_A$ is $\mu$-semistable of degree 0. Then,
$\FM_\kp(F)[1]$ is a sheaf concentrated on a divisor in $|m\rk_C(F)\hat{H}|$.
Thus, the conditions $\mu(E)=0$ and $E$ $\mu$-semistable force
$\FM_\kp(E)[1]$ to be $\mu$-semistable with respect to the dual
polarization $\hat H$.

It remains to show the vanishing of the cohomologies $\FM_\kp(E)^0$ (step 1)
and $\FM_\kp(E)^2$ (step 2) of the complex $\FM_\kp(E)$. After that we prove
that $\FM_\kp(E)^1$ is torsion free (step 3), and locally free (step 4).

Step 1:
If $\FM_\kp(E)^0\ne0$, then we have
 $\Hom(\ko_{\hat A}(-m \hat H), \FM_\kp(E)^0)\ne0$ for $m \gg 0$.
This implies
 $\Hom(\ko_{\hat A}(-m \hat H), \FM_\kp(E))\ne0$
(replace $\FM_\kp(E)$ by a complex concentrated in non-negative
degrees and use the Eilenberg-Moore spectral sequence).
Applying the inverse Fourier-Mukai transform $\FM_\kp\inv$, we get
$\Hom(\FM_\kp\inv(\ko_{\hat A}(-m \hat H)),E)\ne0$.
By \cite[Theorem 2.2]{Mukai}, the inverse is
$\FM_\kp\inv=(-1)^*\FM_{\kp}[2]$. As
$(-1)^*\FM_{\kp}(\ko_{\hat A}(-m \hat H))[2]$ is a semistable vector
bundle with positive first Chern class
(see \cite[Proposition 3.11]{Mukai}), $\FM_\kp(E)^0 \ne 0$ would
contradict the semistability of $E$.

Step 2: Now suppose $\FM_\kp(E)^2\ne0$. We choose a point $P \in
\supp(\FM_\kp(E)^2)$ and obtain a morphism $\FM_\kp(E)^2 \to k(P)$.
As before this gives a morphism $\FM_\kp(E) \to k(P)$, and a morphism
$E \to L_P\inv$ on $A$ where $L_P$ is the line bundle parameterized by
the point $P$. This morphism contradicts the $\mu$-stability of $E$.

Step 3: By what was already proven, we know that $\FM_\kp(E)[1]$ is
$\mu$-semistable. Thus, to show that this sheaf is torsion free, it is
enough to exclude the existence of a subsheaf $T\subset\FM_\kp(E)[1]$
with zero-dimensional support. If $T\ne0$ we have $H^0(T)\ne0$.
We deduce $\Hom(\ko_{\hat A},\FM_\kp(E)[1])\ne0$.
Applying the inverse Fourier-Mukai transform we obtain $\Ext^1(k(0),E)\ne0$.
However, this Ext group vanishes because $E$ was locally free at $0\in A$.
So we derive that $T=0$.

Step 4: Finally we show that the torsion free sheaf $\FM_\kp(E)[1]$ is a
vector bundle. If it was not locally free, there would be a proper
inclusion $\FM_\kp(E)[1] \rarpa{\iota} (\FM_\kp(E)[1])\ddual$. If $P \in
\supp(\coker(\iota))$, then we have $\Ext^1(k(P), \FM_\kp(E)[1])\ne0$, or,
after application of $\FM_\kp\inv$, that $\Hom(L_P\inv,E)\ne0$. But this
contradicts the $\mu$-stability of $E$.
\end{proof}

\begin{localremark}
In the proof of the above theorem the $\mu$-stability of $E$ can be
replaced by the following weaker condition: $E$ is $\mu$-semistable and for all
line bundles $L$ in $\Pic^0(A)$ we have $\Hom(L,E)=\Hom(E,L)=0$.
\end{localremark}

Fix integers $r$ and $s$ and let $\km_A(r,0,s)$ be the moduli space
of $\mu$-semistable sheaves $E$ on $A$ of rank $r$ and $c_1(E)=0$, $c_2(E)=s$.
By Theorem \ref{spanish}, $\FM_P(E)[1]$ is a $\mu$-semistable (and in fact
$\mu$-stable) sheaf for $\mu$-stable $E$. Hence, $\FM_P$ provides an injective
map $U\embed\km_{\hat A}(s,0,r)$ where $U\subset\km_A(r,0,s)$ is the open subset
of $\mu$-stable sheaves. Using the inverse transform $\FM_\kp\inv$ provides a
derived compactification which in the case at hand is nothing but the standard
compactification using $\mu$-semistable sheaves.

\subsection{Reversing universal bundles}

Let $X$ be a smooth, projective variety and $M=M_X(v)$ be a fine moduli space
of sheaves on $X$ with prescribed Mukai vector $v\in H^*(X)$. Denote by $P$
the universal sheaf on $X\times M$ and by $\Phi:=\FM_P:\Db(M)\to\Db(X)$ the
associated Fourier-Mukai transform. The right adjoint is given by
$\Phi^a:=\FM_Q:\Db(X)\to\Db(M)$ with kernel
 $Q=P\dual\otimes p_M^*\omega_M[\dim M]$.
The canonical transformation $\Phi^a\circ\Phi\to\id$ is an isomorphism. This
follows directly from writing $\Phi^a\circ\Phi$ as the Fourier-Mukai transform
whose kernel is the convolution of $P$ and $Q$; cohomology base change shows
that the convolution is a complex concentrated in degree $\dim(X)$, supported on
the diagonal and of rank one there (one can also check that this convolution
is just $\ko_\Delta[\dim(X)]$).

Hence, the adjoint functor $\Phi^a$ is fully faithful. By the Comparison
Theorem, stability on $X$ with parameters $v$ (which by assumption is the
same as semistability) can be phrased as P-stability for a P-datum
$(C_\bullet,N)$ on $X$. By Theorem \ref{theorem-preservation}, $\Phi(E)$ is
P-stable with respect to $(\Phi(C_\bullet),N)$. We also see that $X$
parameterizes such P-stable objects and these are sheaves on $M$ of the same
rank as $E$.

\subsection{Elliptic K3 surface}

Let $\pi:X\to\IP^1$ be an elliptic K3 surface with a section
$\sigma:\IP^1\to X$. Due to the presence of the section, the relative Jacobian
of $\pi$ is isomorphic to $X$ itself. In particular, there is a relative
Poincar\'e bundle $\kp$ on $X\times_{\IP^1}X$. We will use the associated
Fourier-Mukai transform $\Phi:=\FM_\kp:\Db(X)\isom\Db(X)$ which is an equivalence
by standard arguments \cite{huy} or \cite{BBH}.

We have two divisor classes at our disposal: the fibre $f=[\pi\inv(p)]$
(of any point $p\in\IP^1$) and the section $\sigma$. They intersect as $f^2=0$,
$f.\sigma=1$ and $\sigma^2=-2$; the latter because $\sigma\subset X$ is a smooth,
rational curve.

The divisor $H=\sigma+3f$ is big and effective, hence ample as $X$ is a K3 surface. We consider two moduli spaces of $\mu$-semistable sheaves (with respect to $H$) on
$X$. One is the Hilbert scheme $\km_1:=\Hilb^2(X)$ of 0-dimensional subschemes of
length 2 (or rather ideal sheaves of such); it is the moduli space of semistable sheaves of rank 1, $c_1=0$ and $c_2=2$. The other is the moduli space
$\km_2=\km_X(2,-\sigma,0)$ of $\mu$-semistable sheaves with prescribed Chern character. For a decomposable subscheme $Z\subset X$ of length 2 supported on distinct fibres, $\FM_\kp$ maps the twisted ideal sheaf $\ko_X(2\sigma)\otimes\ki_Z$ to a $\mu$-stable sheaf in $\km_2$; see \cite[\S6]{BBH}.

In this way, we obtain an isomorphism between the open set of points of
$\Hilb^2(X)$ with support in different fibres and the locus $\km_2^{s}$ of
stable sheaves. $\FM_P$ also identifies the boundaries. An easy computation
shows that for subschemes $Z$ supported on a single fibre,
$\FM_P(\ko_X(2\sigma)\otimes\ki_Z)$ is a complex with nonzero cohomology
in degrees 0 and 1. In other words, $\FM_P$ provides a compactification of
$\km_2^{s}$ using genuine complexes.

In this roundabout example, the compactification coming from $\km_1$ turns
out to be the same as the classical one by coherent sheaves with a singular
point. We hope that one can still see how P-stability may usefully enter
into the picture.

\subsection{Spherical transforms}

To an object $E\in\kt$ in a (reasonable) $k$-linear triangulated category,
one can associate a canonical functor $\TT_E$
\[ \Hom^\bullet(E,A)\otimes E\to A\to \TT_E(A)\to\Hom^\bullet(E,A)\otimes E[1] . \]
Here, $\Hom^\bullet(E,A)$ is the Hom complex; it is a complex of $k$-vector spaces
whose cohomology in degree $i$ is $\Hom^i(E,A)$. In particular, there is an
isomorphism $\Hom^\bullet(E,A)\cong\bigoplus_i\Hom^i(E,A)[-i]$. The first map
in the triangle is the evaluation map.

Note that due to the non-functoriality of cones, the above triangles are not
enough to define $\TT_E$ on morphisms. There are several ways to rectify this:
if $\kt$ comes from a dg-category, then the construction can be made on the
dg-level and descends to $\kt$. In the geometrical situation, $\kt=\Db(X)$,
one can specify $\TT_E$ as the Fourier-Mukai transform with kernel
 $\ki_E=\Cone(E\dual\boxtimes E\to\ko_\Delta)$,
i.e.\ we choose one cone on the kernel level (for this construction, $E$ has to
be a perfect object).

Assume that $X$ is Gorenstein (or more generally, such that the dualizing complex
$\omega_X$ exists and is perfect). A perfect object $E\in\Db(X)$ is said to be an
\emph{$d$-sphere object} (or $S^d$-object), for some integer $d$, if there is an
isomorphism of graded algebras $\Hom^*(E,E)\cong H^*(S^d,k)$, the only non-vanishing
pieces of latter being one-dimensional in degrees 0 and $d$. For such an object, the
functor $\TT_E:\Db(X)\to\Db(X)$ is fully faithful. By a standard criterion, this can
be checked by testing fully faithfulness on the spanning class $\{E\}\cup E\orth$,
which is easy in view of $\TT_E(E)\cong E[1-d]$ and $\TT_E|_{E\orth}=\id$. (The
assumption on $X$ allows to apply duality, ensuring ${}\orth(\{E\}\cup E\orth)=0$.)
Note that the familiar spherical twist equivalences of Seidel and Thomas 
\cite{Seidel-Thomas} are $\TT_E$ functors for $\dim(X)$-sphere objects satisfying
$E\otimes\omega_X\cong E$.

As an example, consider a ruled surface $X\to E$ over an elliptic curve. Then,
the structure sheaf $\ko_X$ satisfies the above condition with $d=1$ (this
follows from cohomology base change). Hence, we obtain a fully faithful
endofunctor $\TT_{\ko_X}:\Db(X)\to\Db(X)$. Again, this functor can be used to push
forward any P-datum on $X$. For example, if we start with a Hilbert scheme $\Hilb^n(X)$
of points on $X$ and choose a P-datum $(C_\bullet,N)$ describing stability of the ideal
sheaves $\ki_x$, then all $\TT_{\ko_X}(\ki_x)$ are P-stable with respect to
$(\TT_{\ko_X}(C_\bullet),N)$. As $\TT_{\ko_X}$ is not essentially surjective, the two
moduli spaces can differ. However, since $\TT_{\ko_X}$ is fully faithful, it is a
local isomorphism. Thus, the image of a smooth component on one side is a smooth
component on the other.


\end{document}